\documentclass[12pt, a4paper]{article}

\usepackage[utf8]{inputenc}
\usepackage[T1]{fontenc}
\usepackage{amsmath}
\usepackage{amsfonts}
\usepackage{amssymb}
\usepackage{amsthm}
\usepackage{geometry}
\usepackage{hyperref}
\usepackage{mathrsfs}

\hypersetup{
    colorlinks=true,
    linkcolor=blue,
    filecolor=magenta,      
    urlcolor=blue,
    citecolor=blue,
}

\theoremstyle{plain}
\newtheorem{theorem}{Theorem}[section]
\newtheorem{lemma}[theorem]{Lemma}
\newtheorem{conjecture}[theorem]{Conjecture}

\theoremstyle{definition}
\newtheorem{definition}[theorem]{Definition}

\newtheorem*{hypothesis}{Schinzel's H Hypothesis}
\newtheorem*{maintheorem}{Main Theorem}

\title{\textbf{A generalization of the Erdős-Sierpiński conjecture}}

\author{
    Amirali Fatehizadeh \\
    \small \textit{Faculty of Mathematical Sciences, Shahid Beheshti University, Tehran, Iran} \\
    \small \texttt{Amirali.fatehizadeh@gmail.com} \\
    \small \texttt{A.fatehizadeh@mail.sbu.ac.ir}
}
\date{}

\begin{document}
\sloppy

\maketitle

\begin{abstract}
In this paper, we investigate the combinatorial structure and asymptotic distribution of the solution set of the equation $\sigma(n + 1) = k\sigma(n)$ for a given integer $k > 1$. From a combinatorial perspective, the solutions to this equation are closely related to the concept of $k$-layered numbers, which are a generalization of Zumkeller numbers. In the analytic section, which constitutes the core of this research, we employ the framework of probabilistic number theory and an extension of the classical Kubilius model to study the oscillatory and local behavior of the sum-of-divisors function. Utilizing the truncation technique for arithmetic functions and applying the Chinese Remainder Theorem, the problem is reduced to a synthetic measure space equipped with independent random variables. Subsequently, by applying the optimized version of the Kolmogorov-Rogozin anti-concentration inequality (Petrov's theorem) to the difference of additive variables and finely tuning the error parameters, we prove that the natural density of this set is zero. The main quantitative outcome of this approach is the derivation of the explicit upper bound 
$A_k(x) \ll_k \frac{x}{\sqrt{\log \log \log x}}$
for the counting function of the solutions. Finally, alongside the zero asymptotic density, relying on the framework of polynomials and Schinzel's H Hypothesis, we establish the conditional infinitude of the solution set for the case $k = 2$ and formulate the existential results.

\vspace{1em}
\end{abstract}

\section{Introduction}
The study of the structural properties of divisors of an integer and the value distribution of the sum-of-divisors function has roots in ancient Greece. The classification of numbers into perfect, deficient, and abundant based on the values of the proper divisor function (aliquot sum), namely $s(n) = \sigma(n) - n$, was among the earliest steps in this direction. A classical generalization of perfect numbers is the concept of $k$-perfect numbers, for which the equation $\sigma(n) = kn$ holds for a given integer $k > 1$.

In 2003, Zumkeller introduced a novel combinatorial generalization of perfect numbers. A positive integer $n$ is called a Zumkeller number if the set of its positive divisors can be partitioned into two disjoint subsets whose elements sum to $\frac{\sigma(n)}{2}$ respectively. Trivially, all perfect numbers are Zumkeller numbers. Subsequently, researchers in this field, such as Joukar (2019, 2022) and Mahanta, Saikia, and Yaqubi (2020), extended this concept and introduced the notion of $k$-layered numbers. An integer $n$ is termed $k$-layered if its set of divisors can be partitioned into $k$ disjoint subsets with equal sums. Structurally, there are profound connections between these concepts; for instance, every perfect number is necessarily a $2$-layered number, but the converse is not generally true.

Alongside combinatorial properties, studying the local and asymptotic behavior of the function $\sigma(n)$, particularly on consecutive arguments, has always been one of the fundamental challenges in analytic number theory. While the mean behavior of this function is well-understood, its local fluctuations are highly irregular. The celebrated Erdős-Sierpiński conjecture, which asserts the existence of infinitely many solutions to the equation $\sigma(n + 1) = \sigma(n)$, remains a prominent unsolved problem in mathematics.

Over the past decades, extensive efforts have been made to understand similar equations and investigate the density of their solution sets. Erdős, Pomerance, and Sárközy (1987) demonstrated that locally repeated values of certain arithmetic functions possess zero density. It is worth noting that the bound they obtained for $\sigma(n + 1) = \sigma(n)$ is of an exponential order, which is significantly sharper than our bound in the general case; however, our method is robust enough to be applicable to other generalizations of such equations. Furthermore, in various studies, Pomerance and Pollack investigated the density of integers related to amicable pairs and reversals of aliquot sequences, showing that the solutions to equations such as $\frac{\sigma(n)}{n} = \frac{\sigma(m)}{m}$ exhibit upper bounds with very slow growth. In a more recent study, Kobayashi and Trudgian (2020) examined the natural density of the solution set for the inequality $\sigma(n + 1) \geq 2\sigma(n)$. The common thread among all these studies is the establishment of the fact that the coincidence of values of multiplicative functions on consecutive arguments is a rare phenomenon among integers.

Building upon this interplay of combinatorial and analytic concepts in the realm of arithmetic functions, the present paper deals with the asymptotic analysis of a novel generalization of the Erdős-Sierpiński conjecture. This generalization bridges the local behavior of the sum-of-divisors function with the structural concept of $k$-layered and $k$-perfect numbers:
\[ \sigma(n + 1) = k\sigma(n) \]
where $k > 1$ is an integer.

This equation was first introduced and studied in an exploratory paper by Torabi and Fatehizadeh (2020). In that work, adopting an elementary algebraic approach, alongside introducing families of solutions, preliminary results were established under specific conditions. For instance, for the equation $\sigma(n + 1) = 2\sigma(n)$, the only prime solution is $n = 5$.

Nevertheless, the nature of that research was existential. The fundamental question is: does the solution set of this equation possess zero density among the natural numbers? And more importantly, what is the quantitative distribution of the counting function $A_k(x) := \text{\#}\{n \leq x : \sigma(n + 1) = k\sigma(n)\}$? The current paper is conducted with the aim of answering these questions.

To achieve explicit asymptotic bounds, we harness the tools of probabilistic number theory. The cornerstone of this branch rests upon the fundamental idea that the divisibility of an integer by distinct primes exhibits quasi-independent behavior. This idea culminated in the classical Kubilius model, a framework wherein additive arithmetic functions are approximated using the sum of independent random variables over prime numbers.

However, our approach in this paper features subtle extensions and deviations from the standard application of the Kubilius model. Classical models are typically employed to derive global limiting distributions, such as the Erdős-Kac central limit theorem, to describe the overall dispersion of a function's values. In contrast, our problem, after applying a logarithmic transformation to the main equation, concerns examining the behavior of the additive function $g(n) = \log \left(\frac{\sigma(n)}{n}\right)$, and we seek to investigate the probability of the difference $g(n + 1) - g(n)$ falling within a small neighborhood of a target point.

To overcome this challenge, our methodology proceeds as follows: first, using the truncation technique, we restrict the function $g(n)$ to its components supported on small primes. Then, by applying the Chinese Remainder Theorem, we asymptotically approximate the arithmetic distribution by a finite model space. In this modeled space, unlike the original space of integers, the corresponding variables can be rigorously assumed to be independent.

Now, within this modeled space, the distinguishing feature of our method becomes apparent: instead of searching for central limit theorems, we utilize the concept of the Lévy concentration function and anti-concentration inequalities. Relying on the optimized version of the Kolmogorov-Rogozin inequality formulated by Kesten and Petrov, we demonstrate that in this approximated probability model, the variables are widely dispersed along the real line. Consequently, the probability of their concentration around the target value tends to zero at a specific rate.

By combining these advanced analytic and probabilistic tools, the fundamental result of this paper is stated as follows:

\begin{maintheorem}
For any integer $k > 1$ and for sufficiently large $x$, the counting function of the solution set $A_k(x) := \text{\#}\{n \leq x : \sigma(n + 1) = k\sigma(n)\}$ satisfies the following inequality:
\[ A_k(x) \ll_k \frac{x}{\sqrt{\log \log \log x}}. \]
\end{maintheorem}

As a direct consequence of this theorem, it is shown that the natural density of the solution set to the equation $\sigma(n + 1) = k\sigma(n)$ is zero.

The remainder of the paper is structured as follows: In Section 2, we review the preliminaries and the precise formalism of Lévy concentration functions, anti-concentration inequalities, the fundamental concepts of the Kubilius model, and relevant theorems concerning prime numbers. In Section 3, by applying the described methodology, we provide the proof of the main theorem. Finally, in Sections 4 and 5, by examining the discovered solutions, we discuss the structural position of these integers.

Throughout this paper, standard asymptotic notation is utilized. The notation $f(x) \ll g(x)$ or $f(x) = O(g(x))$ means there exists an absolute constant $C > 0$ such that $|f(x)| \leq Cg(x)$. Moreover, the notation $f(x) = o(g(x))$ denotes $\lim_{x \to \infty} \frac{f(x)}{g(x)} = 0$.

\section{Preliminaries}
In this section, we review the notations, combinatorial concepts related to the structure of divisors, basic analytic estimates, and ultimately the tools from probabilistic number theory that form the analytic and algebraic foundation of our research.

The equation studied in this paper is deeply intertwined with combinatorial classifications of integers.

\begin{definition}
A positive integer $n$ is called a $k$-layered number if the set of its divisors can be partitioned into $k$ disjoint subsets such that the sum of the elements in each subset is exactly equal to $\frac{\sigma(n)}{k}$.
\end{definition}

In the case $k = 2$, these integers are renowned as Zumkeller numbers. It is trivial that a necessary condition for being a Zumkeller number is the evenness of $\sigma(n)$ and the inequality $\sigma(n) \geq 2n$; hence, all perfect numbers fall into this category.

\begin{definition}
A positive integer $n$ is called a practical number if every positive integer smaller than $n$ can be expressed as a sum of distinct positive divisors of $n$.
\end{definition}

From a structural perspective, profound connections exist among these classes of integers. Research by Peng and Bhaskara Rao (2013) and studies by Mahanta, Saikia, and Yaqubi (2020) have demonstrated that practical numbers play a generative role in the construction of Zumkeller and $k$-layered numbers. For instance, if $n$ is a practical number satisfying $k \mid \sigma(n)$, and $x$ is a $(k-1)$-layered number coprime to $n$, then their product is a $k$-layered number.

To analytically investigate the equation in question, applying a logarithmic transformation to both sides leads to the definition of an additive arithmetic function that plays a crucial role in our proof.

We define the arithmetic function $g: \mathbb{N} \to \mathbb{R}$ as $g(n) := \log \left(\frac{\sigma(n)}{n}\right)$. Obviously, this is an additive function; meaning that for any two coprime integers $m$ and $n$, we have $g(mn) = g(m) + g(n)$. Its evaluation on prime powers is derived as follows, and we denote it by $\ell$:
\[ \ell(p^a) := \log \left( \frac{1 - p^{-(a+1)}}{1 - p^{-1}} \right). \]

We set $\ell(1) = 0$. Since $\sigma(n) \geq n$ for every natural number $n$, we consistently have $g(n) \geq 0$.

In our computational approximations, we will exploit Taylor series bounds for the logarithm function. The upper bound $\log(1 + x) \leq x$ universally holds for $x \geq 0$. Furthermore, for the lower bound, one can use the Taylor expansion to show that $\log(1 + x) \geq \frac{x}{2}$ holds for $x \in [0, 1)$. Consequently, over the primes, we have:
\[ g(p) = \log(1 + p^{-1}) = p^{-1} + O(p^{-2}), \]
and in particular, $g(p) \sim p^{-1}$.

In addition, to control the asymptotic behavior over prime numbers, we utilize the following two classical theorems, widely known as the Chebyshev bound and Mertens' theorem.
Firstly, the Chebyshev function $\theta(x) = \sum_{p \leq x} \log p$ satisfies the inequality $\theta(x) < 1.02x$ for sufficiently large $x$. Secondly, the sum of the reciprocals of primes exhibits the asymptotic behavior $\sum_{p \leq x} p^{-1} = \log \log x + B + o(1)$, which implies $\sum_{p \leq x} \frac{2}{p} \sim 2 \log \log x$.

On the other hand, for the algebraic existential investigations of this equation, we invoke one of the deepest conjectures in number theory.

\begin{hypothesis}
Let $P_1(x), P_2(x), \dots, P_m(x)$ be irreducible polynomials over the ring of integers with positive leading coefficients. Suppose there is no prime number that divides their product for all integers $x$. Then, there exist infinitely many positive integers $n$ such that the values $P_1(n), P_2(n), \dots, P_m(n)$ are all simultaneously prime.
\end{hypothesis}

To find explicit bounds and prove the zero density, we need to formulate the problem within the framework of probabilistic number theory.

Let the sample space be $\Omega = \{1, 2, \dots, N\}$ and $\Sigma$ be the $\sigma$-algebra of all subsets of $\Omega$. For any $A \in \Sigma$, we define the probability measure as $\nu_N(A) = \frac{\text{\#}A}{N}$. Consequently, this forms a uniform discrete probability space. The fundamental challenge in studying additive functions is that arithmetic events related to divisibility are merely quasi-independent and lack absolute independence.

The Fundamental Lemma of Kubilius resolves this challenge by providing a probabilistic modeling in a synthetic measure space. In this modeled space with probability measure $\mathbb{P}$, the $p$-adic valuation arithmetic variable, namely $v_p(n)$, is simulated by random variables $Z_p$ that are mutually independent, and their distribution perfectly matches the limiting distribution of $v_p(n)$ over the natural numbers:
\[ \mathbb{P}(Z_p = a) = \left(1 - \frac{1}{p}\right) \frac{1}{p^a}. \]

However, for the approximation of the arithmetic space by the independent space to be mathematically valid, the target additive function must not be evaluated over all prime divisors; rather, it must be truncated up to a specific prime bound, say $y$. In the Kubilius model, the truncation parameter $y$ relative to the scale parameter $N$ must be chosen such that $\frac{\log y}{\log N} \to 0$.

Under this condition, sieve theory ensures that the error in replacing the distribution of the truncated arithmetic variable $\sum_{p \leq y} g(v_p(n))$ with the distribution of the sum of independent random variables $\sum_{p \leq y} g(Z_p)$---measured as a total variation distance---is strictly controlled and tends to zero.

This rigorous framework allows us to exploit classical tools of independent probabilities to analyze quasi-independent arithmetic structures.

Now that the arithmetic behavior is reduced to a sum of independent random variables via the Kubilius approximation, the ultimate challenge is to control the concentration function of the additive variable $g(n + 1) - g(n)$ in the neighborhood of the target value. Central limit theorems merely describe the global distribution of the variables; however, to bound the maximum probability of the sum of variables falling into a small interval, we require anti-concentration tools.

\begin{definition}
For a random variable $X$ and a scale of length $L > 0$, the Lévy concentration function is defined as:
\[ Q_L(X) := \sup_{a \in \mathbb{R}} \mathbb{P}(X \in [a, a + L]). \]
\end{definition}
This function measures the maximum probability mass that the random variable $X$ can concentrate in any arbitrary interval of length $L$.

To control this concentration in the sum of independent random variables within the Kubilius model, we employ an optimized version of the Kolmogorov-Rogozin inequality, extended by Kesten and Petrov:

\begin{theorem}
Let $X_1, X_2, \dots, X_M$ be independent random variables and $W = \sum_{j=1}^M X_j$. For any global scale parameter $L > 0$, if there exist local scales $L_j$ such that $0 < L_j \leq L$ holds for every $j$, then there exists an absolute constant $A > 0$ such that
\[ Q_L(W) \leq A \cdot L \cdot \left( \sum_{j=1}^M L_j^2 \left(1 - Q_{L_j}(X_j)\right) \right)^{-1/2}. \]
\end{theorem}

This foundational theorem demonstrates that the independent variables $Z_p$ in the modeled space cannot concentrate their probability mass in small intervals; consequently, the probability of them falling into the desired neighborhood tends to zero at an explicit rate.

\section{Main Results}
In this section, by transitioning to a periodic probability space and optimizing the parameters, we provide an explicit upper bound for the counting function of the solutions, and as a direct consequence, we prove that the natural density of $A_k$ is zero. Before stating and proving the main theorem, it is worth noting that, using Theorem 1.1 in the paper by Mangerel, it easily follows that the set $A_k$ has logarithmic density zero.

\begin{theorem}
Let $k > 1$ be an integer number. For sufficiently large $x$, we have
\[ A_k(x) \ll_k \frac{x}{\sqrt{\log \log \log x}}. \]
\end{theorem}

\begin{proof}
Let $n \in A_k$. By dividing both sides of the main equation by $n(n + 1)$ and applying the logarithmic function, the equation is algebraically equivalent to the following form:
\[ g(n + 1) - g(n) = \log k - \log \left(1 + \frac{1}{n}\right). \]

To control the behavior of the function $g$, we truncate it using two constant parameters $y \geq 2$ and $r \geq 1$. We define the truncated function as $h_{y,r}(n) := \sum_{p \leq y} \ell\left(p^{\min{(v_p(n), r)}}\right)$ and its consecutive difference as $D_{y,r}(n) := h_{y,r}(n + 1) - h_{y,r}(n)$.

For a real number $x > 0$ and an error parameter $\epsilon > 0$, we construct the following sets in the interval $[1, x]$:
\begin{enumerate}
    \item $E_0(x, \epsilon) := \left\{n \leq x : \log \left(1 + \frac{1}{n}\right) > \epsilon \right\}$
    \item $E_1(x, y, \epsilon) := \{n \leq x : g_{>y}(n) > \epsilon \lor g_{>y}(n + 1) > \epsilon\}$
    \item $E_2(x, y, r) := \{n \leq x : \exists p \leq y, v_p(n) > r \lor v_p(n + 1) > r\}$
    \item $S(x, y, r, \epsilon) := \{n \leq x : |D_{y,r}(n) - \log k| \leq 3\epsilon\}$
\end{enumerate}

Suppose $n \leq x$ is a solution to the equation that does not belong to any of the sets $E_0, E_1$, and $E_2$. Since $n \notin E_2$, the decomposition $g(n) = h_{y,r}(n) + g_{>y}(n)$ is valid for both $n$ and $n + 1$. By substituting this into the main equation and applying the triangle inequality, the conditions $n \notin E_0$ and $n \notin E_1$ imply that $|D_{y,r}(n) - \log k| \leq 3\epsilon$. Therefore, the set inclusion $A_k \cap [1, x] \subseteq E_0 \cup E_1 \cup E_2 \cup S$ strictly holds, and we have:
\begin{equation}
A_k(x) \leq \text{\#}E_0(x, \epsilon) + \text{\#}E_1(x, y, \epsilon) + \text{\#}E_2(x, y, r) + \text{\#}S(x, y, r, \epsilon). \label{eq:1}
\end{equation}

Now, we calculate an explicit upper bound for each of these error sets.

For the bound of $E_0$, condition $\log \left(1 + 1/n\right) > \epsilon$ implies $n < 1/\epsilon$. Thus, $\text{\#}E_0(x, \epsilon) \leq \frac{1}{\epsilon}$.

For the bound of $E_1$, considering the non-negativity of the terms $f_y(n) := g_{>y}(n) = \sum_{p > y} \ell(p^{v_p(n)})$, for any prime $p$ and integer $a \geq 1$, we have:
\[ \ell(p^a) = \log \left( \frac{1 - p^{-(a+1)}}{1 - p^{-1}} \right) \leq \log \left(\frac{1}{1 - p^{-1}}\right) = \log \left(\frac{p}{p - 1}\right) \leq \frac{2}{p}. \]

On the other hand, $\text{\#}\{n \leq x + 1 : v_p(n) = a\} \leq \lfloor \frac{x+1}{p^a} \rfloor \leq \frac{x+1}{p^a}$. Consequently, we obtain:
\[ \sum_{n \leq x+1} f_y(n) \leq (x + 1) \sum_{p > y} \sum_{a \geq 1} \frac{2}{p^{a+1}} = 2(x + 1) \sum_{p > y} \frac{1}{p(p - 1)}. \]

Now, using the Prime Number Theorem and partial summation, we have $\sum_{n \leq x+1} f_y(n) \ll \frac{x}{y \log y}$. Applying Markov's inequality to this non-negative function and considering the union bound for the two events associated with $n$ and $n + 1$, we get:
\[ \text{\#}E_1(x, y, \epsilon) \ll \frac{x}{\epsilon \cdot y \log y}. \]

For the bound of $E_2$, the number of integers up to $x$ that are divisible by $p^{r+1}$ is at most $\lfloor \frac{x}{p^{r+1}} \rfloor$. Using the integral test for the series over all integers $n \geq 2$, we have:
\[ \sum_{p \leq y} \frac{1}{p^{r+1}} \leq \sum_{n=2}^{\infty} \frac{1}{n^{r+1}} \leq \frac{1}{2^{r+1}} + \int_{2}^{\infty} \frac{dt}{t^{r+1}} = \frac{1}{2^{r+1}} + \frac{1}{r \cdot 2^r} \ll 2^{-r}. \]
Consequently, $\text{\#}E_2(x, y, r) \ll x \cdot 2^{-r}$.

Now, the function $D_{y,r}(n)$ over the natural numbers is strictly periodic, with a period of $M_{y,r} := \prod_{p \leq y} p^r$. For sufficiently large $x$, the parameter $y$ will also be sufficiently large. Thus, by the Prime Number Theorem and Chebyshev's inequality, we can write $\log M_{y,r} = r\theta(y) < 1.02ry$, and consequently $M_{y,r} \leq \exp(1.02ry)$.

We divide the interval $[1, x]$ into $m = \lfloor x / M_{y,r} \rfloor$ full periods and a remaining part of length strictly less than $M_{y,r}$. We equip the finite space $\mathbb{Z}/M_{y,r}\mathbb{Z}$ with a uniform probability measure. By the Chinese Remainder Theorem, the following ring isomorphism holds:
\[ \mathbb{Z}/M_{y,r}\mathbb{Z} \cong \prod_{p \leq y} \mathbb{Z}/p^r\mathbb{Z}. \]

Under this isomorphism, the uniform measure on the left-hand finite space is transferred to the product of uniform measures on the right-hand spaces. Consequently, if the random variable $N$ follows a uniform distribution on $\mathbb{Z}/M_{y,r}\mathbb{Z}$, its components $N_p := N \pmod{p^r}$ will be mutually independent random variables uniformly distributed on $\mathbb{Z}/p^r\mathbb{Z}$. Therefore, the distribution of $D_{y,r}(n)$ over each full period is exactly identical to the distribution of the following random variable:
\[ W_{y,r} := \sum_{p \leq y} Z_{p,r}, \quad Z_{p,r} := \ell\left(p^{\min(v_p(N_p + 1), r)}\right) - \ell\left(p^{\min(v_p(N_p), r)}\right). \]

Since the remaining part introduces an error bounded by its own length, $M_{y,r}$, we obtain:
\begin{equation}
\text{\#}S(x, y, r, \epsilon) \leq x \cdot \mathbb{P}(|W_{y,r} - \log k| \leq 3\epsilon) + \exp(1.02ry). \label{eq:2}
\end{equation}

Now, to apply Petrov's theorem, we first formulate the structure of the variables $Z_{p,r}$ in the form of the following lemma.

\begin{lemma}
Suppose $p \geq 5$ and the length of a closed interval $L$ is $L$. If $L < \ell(p)$, then for every $r \geq 1$ we have:
\[ Q_L(Z_{p,r}) = 1 - \frac{2}{p}. \]
\end{lemma}

\begin{proof}
The variables $N_p$ and $N_p + 1$ are consecutive; thus, they cannot simultaneously be multiples of $p$. This restricts the support of the variable $Z_{p,r}$ to three disjoint clusters:
\begin{enumerate}
    \item \textbf{Zero cluster:} $\mathbb{P}(Z_{p,r} = 0) = \mathbb{P}(p \nmid N_p(N_p + 1)) = 1 - 2/p$.
    \item \textbf{Positive cluster:} If $p \mid N_p + 1$. The total probability mass of this cluster is $1/p$, and its values fall within the interval $[\ell(p), \ell(p^r)]$; we have $\mathbb{P}(Z_{p,r} > 0) = 1/p$.
    \item \textbf{Negative cluster:} If $p \mid N_p$. Similar to the positive cluster, the total mass is $1/p$, and it takes values within the interval $[-\ell(p^r), -\ell(p)]$; we have $\mathbb{P}(Z_{p,r} < 0) = 1/p$.
\end{enumerate}

Therefore, we have $\text{supp}(Z_{p,r}) \subset \{0\} \cup \{\pm \ell(p^a)\}_{a=1}^r$. Since the minimum geometric distance between zero and any non-zero point is $\ell(p)$, and $L < \ell(p)$, no closed interval of length $L$ can simultaneously cover more than one component of the support. Thus, the maximum mass occurs exactly at zero. This completes the proof of Lemma 3.2.
\end{proof}

We now define the set of effective primes as $\mathcal{P}^* := \{p \mid 5 \leq p \leq \min(y, \frac{1}{13\epsilon})\}$. For any $p$ in this set, utilizing the bound $\log(1 + x) \geq x/2$, we have:
\[ \ell(p) = \log \left(1 + \frac{1}{p}\right) \geq \frac{1}{2p} \geq 6.5\epsilon > 6\epsilon. \]

Given the conditions of the preceding lemma, by setting the scale $L = 6\epsilon$, for all $p \in \mathcal{P}^*$ we obtain $Q_{6\epsilon}(Z_{p,r}) = 1 - 2/p$.

We decompose the independent random variable $W_{y,r}$ as $W_{y,r} = S^* + S'^*$, where $S^* := \sum_{p \in \mathcal{P}^*} Z_{p,r}$. Since the convolution with an independent random variable can only decrease the concentration function, we have $Q_{6\epsilon}(W_{y,r}) \leq Q_{6\epsilon}(S^*)$. Now, we apply Petrov's inequality (discussed in the previous section) on $S^*$ with both global and local scales set to $L = L_p = 6\epsilon$. By substitution, we obtain:
\[ Q_{6\epsilon}(W_{y,r}) \leq A \cdot 6\epsilon \cdot \left( \sum_{p \in \mathcal{P}^*} (6\epsilon)^2 (2/p) \right)^{-1/2}. \]

Upon simplification, the following inequality is deduced:
\begin{equation}
\mathbb{P}(|W_{y,r} - \log k| \leq 3\epsilon) \leq A \left( \sum_{p \in \mathcal{P}^*} \left(\frac{2}{p}\right) \right)^{-1/2}. \label{eq:3}
\end{equation}

Combining inequalities (\ref{eq:1}), (\ref{eq:2}), and (\ref{eq:3}), for all sufficiently large $x$, we have:
\[ \frac{A_k(x)}{x} \leq \frac{1}{x\epsilon} + \frac{C_1}{\epsilon \cdot y \log y} + C_2(2^{-r}) + \frac{\exp(1.02ry)}{x} + A \left( \sum_{p \in \mathcal{P}^*} \left(\frac{2}{p}\right) \right)^{-\frac{1}{2}}. \]

We now fix the parameters as functions of $x$ as follows:
\[ r(x) := \lfloor \log \log \log x \rfloor, \quad y(x) := \frac{\log x}{3 \log \log x}, \quad \epsilon(x) := \frac{1}{13y(x)}. \]

The explicit asymptotic evaluation of each term as $x \to \infty$ is as follows:

For the Petrov term, given the choice of $\epsilon$, the upper bound for the primes in the effective set is exactly $y$. Thus, $\mathcal{P}^* = \{5 \leq p \leq y\}$. Since excluding a few constant primes does not affect the asymptotic behavior, Mertens' theorem yields:
\[ \sum_{p \in \mathcal{P}^*} \left(\frac{2}{p}\right) \sim 2 \log \log y \sim 2 \log \log \log x \]
Thus, this term is of order $O((\log \log \log x)^{-1/2})$.

For the $E_1$ term, substitution and simplification yield $\frac{13}{\log y} \sim \frac{13}{\log \log x}$. On the other hand, $(\log \log x)^{-1} = o((\log \log \log x)^{-1/2})$. Hence, this term is asymptotically smaller than the Petrov term.

For the $E_2$ term, we have $2^{-r} \asymp (\log \log x)^{-\log 2}$. Since  $\log 2 > 0.5$ and $(\log \log x)^{-\log 2} = o((\log \log \log x)^{-1/2})$, this term is also dominated by the Petrov term.

For the periodic error term, since $ry = o(\log x)$, for sufficiently large $x$ we always have $\exp(1.02ry) \leq x^{1/2}$. Consequently, this term is bounded by $x^{-1/2}$ and thus by $o((\log \log \log x)^{-1/2})$.

For the $E_0$ term, since it equals $\frac{13y}{x}$, it is strictly $o(1)$ compared to the logarithmic terms.

Since all error terms are strictly of an order smaller than $(\log \log \log x)^{-1/2}$, they are all absorbed by the Petrov term. Consequently, the asymptotic upper bound is exclusively governed by the behavior of the concentration function, and we strictly obtain:
\[ A_k(x) \ll_k \frac{x}{\sqrt{\log \log \log x}} \]

This concludes the proof.
\end{proof}

The upper asymptotic density of the solution set is defined as:
\[ \bar{d}(A_k) = \limsup_{x \to \infty} \frac{A_k(x)}{x}. \]

From Theorem 3.1, we deduce:
\[ \bar{d}(A_k) \leq \lim_{x \to \infty} \frac{C_k}{\sqrt{\log \log \log x}} = 0. \]

Since density is a non-negative quantity, it strictly follows that $\bar{d}(A_k) = 0$. Therefore, the natural density of the solution set for the equation $\sigma(n + 1) = k\sigma(n)$ is zero.

\section{A Family of solutions}
In Section 2, we observed the structural connections of this equation with $k$-layered numbers. We also noted that in the special case $k = 2$, these integers correspond to Zumkeller numbers. Furthermore, in Section 3, we demonstrated the asymptotic rarity of the solutions to this equation. In this section, we intend to first perform a computational search to find several initial solutions for small values of $k$, establishing that the solution sets for these initial values are not empty. Subsequently, invoking Schinzel's H Hypothesis, we will prove that the equation $\sigma(n + 1) = 2\sigma(n)$ has infinitely many solutions. We will then formulate these observations into a generalized conjecture.

Through a computational search, we found a short list of solutions for $k \in \{2, 3\}$ up to a bounded interval. For $k = 2$, the solution set is
\[ \{5, 125, 1253, 1673, 3127, 5191, 7615, 12035\}, \]
and for $k = 3$, the set is
\[ \{1, 1919, 2759, 11219\}. \]

\begin{theorem}
Assuming the truth of Schinzel's H Hypothesis, the equation $\sigma(n + 1) = 2\sigma(n)$ possesses infinitely many solutions.
\end{theorem}

\begin{proof}
Let $n(x) = (252x + 223)(6x + 5)$ and $n(x) + 1 = 6(7x + 6)(36x + 31)$.

A straightforward calculation yields $1512x^2 + 2598x + 1115$ and $1512x^2 + 2598x + 1116$, respectively. Hence, their consecutiveness is evident. Moreover, through simple calculations and utilizing the values of the sum-of-divisors function for prime numbers, it becomes clear that if these linear polynomials simultaneously take prime values for an integer $x$, then the equation $\sigma(n(x) + 1) = 2\sigma(n(x))$ holds. Now, we verify the conditions of Schinzel's H Hypothesis.

Let
\[ P(x) = (252x + 223)(6x + 5)(7x + 6)(36x + 31). \]

Note that $P(1) = 4550025$ and $P(2) = 25459480$. Therefore, the greatest common divisor of these values is $5$. By evaluating the factors of $P(x)$ at $x = 3$, we deduce that none of them are divisible by $5$. Hence, there is no fixed prime divisor for all values of $P(x)$, and the conditions of Schinzel's H Hypothesis are satisfied. Thus, the theorem is proved.
\end{proof}

With all these observations, we formulate this generalization of the Erdős-Sierpiński conjecture as follows.

\begin{conjecture}
For any positive integer $k$, the equation $\sigma(n + 1) = k\sigma(n)$ has infinitely many solutions.
\end{conjecture}

\section{Discussion and Conclusion}
In this paper, the equation $\sigma(n + 1) = k\sigma(n)$ was studied from two perspectives: asymptotic analysis and the existence of solutions. In this concluding section, alongside reviewing the main results, we take an intuitive and observational look at the structural placement of the solutions of this equation among the known classes in number theory, which may pave the way for future research.

The central theme of this paper was studying the asymptotic behavior and density of the set $A_k$. By transitioning from classical methods to a probabilistic space, and harnessing the Kolmogorov-Rogozin anti-concentration inequality, we successfully proved the explicit upper bound $A_k(x) \ll_k x / \sqrt{\log \log \log x}$. This bound establishes that the natural density of the solution set is zero for any $k > 1$.

Alongside this result, the existential investigation of the solutions demonstrated that despite the zero density, this set is not empty. Relying on the polynomial framework and Schinzel's H Hypothesis, the conditional infinitude of solutions for the case $k = 2$ was proved, which led to the formulation of Conjecture 4.2 as a generalization of the Erdős-Sierpiński conjecture for all natural values of $k$.

Although our primary results were of an analytic and asymptotic nature, computational explorations and the examination of the discovered families of solutions reveal a network of deep, intuitive connections with several significant classes of integers. From an observational standpoint, the behavior of the component $N = n + 1$ in this equation has meaningful connections with the following categories.

\textbf{Zumkeller and $k$-layered numbers:} An integer $N$ is a $k$-layered number if its set of divisors can be partitioned into $k$ disjoint subsets such that the sum of each subset equals $\sigma(N)/k$ (the case $k = 2$ corresponds to Zumkeller numbers). A necessary algebraic condition for this partition is the congruence $\sigma(N) \equiv 0 \pmod{k}$, a condition that our equation explicitly satisfies. Our computational evidence strongly suggests that the non-trivial solutions to this equation generally possess this combinatorial partition property. This implies that the solutions to our equation practically avoid falling into the category of ``weird numbers''---abundant numbers for which the subset sum partition of divisors fails.

\textbf{Practical and abundant numbers:} The intuitive success of our equation's solutions in forming $k$-layered structures is rooted in the nature of their factorization. Numerical investigations show that for many solutions, the value $N$ has a large number of small prime divisors (resulting from higher powers of small primes). This characteristic frequently places the solutions in the category of practical numbers. The property of being a practical number provides the intuitive guarantee that the summation of subsets to achieve $k$-layers will be successful. Furthermore, for $k > 1$, our equation clearly necessitates that $N$ is classified as an abundant number.

\textbf{$k$-perfect numbers:} As established in previous research regarding this equation, whenever the variable $n$ is a prime number, satisfying the equation is equivalent to the fact that the integer $n + 1$ is exactly a $k$-perfect number. This algebraic link indicates that the search for prime solutions to our equation coincides with the search for $k$-perfect numbers that satisfy the specific properties of the aforementioned equation.

\renewcommand{\refname}{References}

\end{document}